\documentclass[11pt]{article}  
\usepackage{amsmath}
\usepackage{amssymb}
\usepackage{theorem}
\usepackage{euscript}
\usepackage{pstricks}
\usepackage{authblk}
\usepackage{verbatim}
\usepackage{graphics}
\usepackage{graphicx}
\usepackage{tikz}
\usepackage{amscd}
\usepackage{color}

\usepackage{hyperref}
\hypersetup
{
	colorlinks,
	citecolor=blue,
	filecolor=green,
	linkcolor=green,
	urlcolor=blue
}
\hypersetup{linktocpage}

\newtheorem{theorem}{Theorem}[section]
\newtheorem{lemma}[theorem]{Lemma}

\numberwithin{equation}{section}
\newtheorem{prop}[theorem]{Proposition}
\theoremstyle{definition}
\newtheorem{definition}[theorem]{Definition} 
\theoremstyle{remark}

\newcommand{\brac}[1]{\left(#1\right)}

\newcommand{\bk}{{\boldsymbol{k}}}
\newcommand{\bell}{{\boldsymbol{\ell}}}

\newcommand{\bx}{{\boldsymbol{x}}}

\newcommand{\bX}{{\boldsymbol{X}}}
\newcommand{\by}{{\boldsymbol{y}}}

\newcommand{\balpha}{{\boldsymbol{\alpha}}}

\newcommand{\rd}{{\rm d}}

\newcommand{\mix}{{\rm m}} 
\newcommand{\rank}{{\rm rank}} 
\def\ZZd{{\mathbb Z}^d}
\def\IId{{\mathbb I}^d}

\def\RR{{\mathbb R}}
\def\RRd{{\mathbb R}^d}

\def\NN{{\mathbb N}}
\def\NNd{{\NN}^d}

\def\NN{{\mathbb N}}
\def\RR{{\mathbb R}}

\def\UU{{\mathbb U}}
\def\UUd{{\mathbb U}^d}

\def\IId{{\mathbb I}^d}

\def\NNd{{\mathbb N}^d}
\def\RRd{{\mathbb R}^d}

\def\ZZd{{\mathbb Z}^d}

\def\NN{{\mathbb N}}
\def\RR{{\mathbb R}}

\def\IId{{\mathbb I}^d}

\def\NNd{{\mathbb N}^d}
\def\RRd{{\mathbb R}^d}

\def\supp{\operatorname{supp}}

\def\ext{\operatorname{ext}}
\def\Wpgamma{W^s_p(\RRd,\gamma)}
\def\Wap{W^s_p}

\def\Lqgamma{L_q(\RRd,\gamma)}

\newcommand{\norm}[2]{\left\|{#1}\right\|_{#2}}

\title{\sffamily {Widths of embeddings of Gaussian Sobolev spaces}}

\author{Van Kien Nguyen}
\affil{Department of Mathematical Analysis, University of Transport and Communications
	\protect\\	No.3 Cau Giay Street, Lang  Ward,
	Hanoi, Vietnam
	\protect\\
	Email: kiennv@utc.edu.vn}

\date{\today}
\begin{document}
\maketitle

\begin{abstract}
In this paper, we investigate the approximation problem for functions in Gaussian Sobolev spaces $W^s_p(\mathbb{R}^d, \gamma)$ of smoothness $s > 0$, where the approximation error is measured in the Gaussian Lebesgue space $L_q(\mathbb{R}^d, \gamma)$. Such function spaces naturally arise in the analysis of high-dimensional problems with Gaussian measures and play an important role in various applications, including uncertainty quantification and stochastic modeling.
Our main objective is to analyze the asymptotic behavior of fundamental quantities that characterize the complexity of the approximation problem. In particular, we determine the exact asymptotic order of several classes of widths, including Kolmogorov, linear, and sampling widths, which quantify the optimal performance of different approximation methods. The obtained results cover the parameter regimes $1 \leq q < p < \infty$ and $p = q = 2$, where distinct phenomena in terms of approximation rates can be observed.
	
	
	
\end{abstract}

\section{Introduction}
 \label{Introduction}

In recent years, there has been a rapidly growing body of research devoted to multivariate approximation in the framework of Gaussian Sobolev spaces \cite{DN23,DILP18,Nguyen2025,DDung2025,DDung2026}. This line of research is strongly motivated by the observation that many problems arising in fields such as mathematical finance, statistical physics, quantum chemistry, and machine learning are inherently high-dimensional and naturally formulated on domains equipped with Gaussian measures. In such settings, the functions of interest typically exhibit a certain degree of regularity that can be described in terms of Sobolev smoothness with respect to the Gaussian measure. Consequently, the study of approximation problems in Gaussian Sobolev spaces is both natural and of fundamental importance.

A central question in this context is to understand the intrinsic difficulty of multivariate approximation problems and to characterize the optimal performance of numerical methods. In a recent paper \cite{DN23}, Dinh Dũng and the present author investigated linear approximation and sampling recovery for functions in the Gaussian Sobolev space $W^s_{p,\mix}(\RRd,\gamma)$ of mixed smoothness $s \in \NN$, with the error measured in $L_q(\RR^d,\gamma)$. The asymptotic optimality of these methods was established in terms of Kolmogorov, linear, and sampling widths, which serve as fundamental measures of complexity in approximation theory. More precisely, the exact asymptotic orders were obtained for the regimes $1 \leq q < p < \infty$ and $p = q = 2$, thereby providing a comprehensive characterization of the behavior of optimal approximation schemes in these parameter ranges.

The main objective of this paper is to analyze the asymptotic behavior of Kolmogorov, linear, and sampling widths for functions in Gaussian Sobolev spaces $W^s_p(\mathbb{R}^d, \gamma)$ of smoothness $s > 0$, where the approximation error is measured in the Gaussian Lebesgue space $L_q(\mathbb{R}^d, \gamma)$.
Our approach is based on a decomposition technique developed in \cite{DN23}, combined with refined estimates of the kernel associated with fractional Sobolev spaces. The obtained results cover the parameter regimes $s > 0$ and $1 \leq q < p < \infty$, as well as the case $p = q = 2$, where distinct behaviors of approximation rates can be observed. More precisely, in the case $1 \leq q < p < \infty$ and $s > 0$, we establish for the Kolmogorov and linear widths that
	\begin{equation*}
	\begin{aligned}
		d_n\big({\boldsymbol{W}}^s_{p}(\mathbb{R}^d,\gamma),L_q(\RRd,\gamma )\big)
		& \asymp \lambda_n\big({\boldsymbol{W}}^s_{p}(\mathbb{R}^d,\gamma),L_q(\RRd,\gamma) \big)
	 	\asymp n^{-s/d}. 
	\end{aligned} 
\end{equation*}
and for sampling widths if $s>d/p$
	\begin{equation*}
	\begin{aligned}
	 \varrho_n\big({\boldsymbol{W}}^s_{p}(\mathbb{R}^d,\gamma),L_q(\RRd,\gamma) \big)	\asymp n^{-s/d}. 
	\end{aligned} 
\end{equation*}
In the case $s>0$ and $p=q=2$ we get
	\begin{equation*} 
	d_n\big({\boldsymbol{W}}^s_{2}(\mathbb{R}^d,\gamma),L_2(\RRd,\gamma )\big)
	= \lambda_n\big({\boldsymbol{W}}^s_{2}(\mathbb{R}^d,\gamma),L_2(\RRd,\gamma) \big)
	\asymp 
	n^{-\frac{s}{2d}}.
\end{equation*}
Moreover, in this paper we also obtain the optimal convergence rates for the Kolmogorov and linear $n$-widths of fractional Gaussian Sobolev spaces of mixed smoothness $W^s_{p,\mix}(\mathbb{R}^d,\gamma)$ in the range $2 \leq p < \infty$, $1 \leq q < p$, $s > 0$, and $s \notin \mathbb{N}$
\begin{equation*}
		d_n\big({\boldsymbol{W}}^s_{p,\mix}(\mathbb{R}^d,\gamma),L_q(\RRd,\gamma )\big)
 \asymp \lambda_n\big({\boldsymbol{W}}^s_{p,\mix}(\mathbb{R}^d,\gamma),L_q(\RRd,\gamma) \big)
\asymp n^{-s}(\log n)^{(d-1)(s+\frac{1}{2}-\frac{1}{p})}. 
\end{equation*}

The paper is organized as follows. In Section~\ref{Numerical integration}, we introduce Gaussian Sobolev spaces and fractional Gaussian Sobolev spaces, and discuss their fundamental properties. Section~\ref{sec:numer} is devoted to the estimation of Kolmogorov, linear and sampling $n$-widths for embeddings of  Gaussian Sobolev spaces. In Section~\ref{sec:sec4}, we extend this analysis to the case of fractional Gaussian Sobolev spaces of mixed smoothness, deriving corresponding results for Kolmogorov and linear $n$-widths.
%%%%%%%%%%%%%%%%%%%%%%%%%%%%%%%%%%

\noindent
{\bf Notation.}  
The letter $d$ is always used to denote the underlying dimension of $\RR^d$, $\NN^d$, etc. Vectors in $\RR^d$ are denoted by boldface letters. For $\bx \in \RR^d$, we write $\bx := (x_1, \ldots, x_d)$, where $x_i$ denotes the $i$th coordinate of $\bx$.
For $1 \le p < \infty$, we define $|\bx|_p := \big(\sum_{i=1}^d |x_i|^p\big)^{1/p}$ and $|\bx|_\infty := \max{|x_1|, \ldots, |x_d|}$. In the case $p = 2$, we simply write $|\bx|$.
Let $A_n$ and $B_n$ be quantities depending on $n$ in an index set $J$. We write $A_n \ll B_n$ if there exists a constant $C > 0$, independent of $n$, such that $A_n \le C B_n$ for all $n \in J$. We write $A_n \asymp B_n$ if both $A_n \ll B_n$ and $B_n \ll A_n$ hold.
General positive constants are denoted by $C$, while constants depending on parameters $s, d, \ldots$ are denoted by $C_{s,d,\ldots}$. The values of these constants may vary from line to line.
For a finite set $G$, we denote by $|G|$ its cardinality. The unit ball in a Banach space $X$ is denoted by $\bX$.

%%%%%%%%%%%%%%%%%%%%%%%%%%%%%%%%%
%%%%%%%%%%%%%%%%%%%%%%%%%%%%%%%%%

\section{Gaussian and fractional Gaussian Sobolev spaces}
\label{Numerical integration}

\label{Subsec-AssemblingQuadratures}
In this section, we introduce Gaussian and fractional Gaussian Sobolev spaces defined in the sense of Gagliardo as well as via a kernel associated with the fractional Ornstein–Uhlenbeck operator. We then investigate the relationship between these two types of spaces and establish some of their fundamental properties.

Let $\rd\gamma (\bx) = \rho(\bx) \rd\bx$  be the $d$-dimensional standard Gaussian measure on $\RRd$  with the density 
$$
\rho(\bx):=(2\pi)^{-d/2} \exp\brac{-|\bx|^2/2},\ \ \bx\in \RRd.
$$
Let  $1\leq p<\infty$.  
	We define the Gaussian Lebesgue space  $L_p(\RRd,\gamma)$ to be the set of all measurable functions $f$ on $\RRd$ such that the norm
	$$
	\|f\|_{L_p(\RRd,\gamma)} : = \bigg( \int_{\RRd} |f(\bx)|^p \rd \gamma(\bx)\bigg)^{1/p}
	=
	\bigg( \int_{\RRd} |f(\bx)|^p \rho(\bx) \rd \bx\bigg)^{1/p} \ <  \ \infty. 
	$$
For $s \in \mathbb{N}$, we define the Gaussian Sobolev space $W^s_p(\mathbb{R}^d,\gamma)$ as the normed space consisting of all functions $f \in L_p(\mathbb{R}^d,\gamma)$ such that the generalized partial derivatives $D^\balpha f$ of order $\balpha$ belong to $L_p(\mathbb{R}^d,\gamma)$ for all multi-indices $\balpha \in \mathbb{N}_0^d$ with $|\balpha|_1 \le s$. The norm of a function $f$ in this space is defined by
	\begin{equation}  \label{eq:sobolev}
		\|f\|_{W^s_{p}(\RRd,\gamma)}: = \Bigg(\sum_{|\balpha|_1 \leq s} \|D^\balpha f\|_{L_p(\RRd,\gamma)}^p\Bigg)^{1/p}.
	\end{equation}

Before introducing fractional Gaussian Sobolev spaces on $\mathbb{R}^d$, we  define the classical fractional Sobolev spaces on $\mathbb{U}^d$, where $\mathbb{U} = [a,b]$ or $\mathbb{U} = \mathbb{R}$. As usual, $L_p(\mathbb{U}^d)$ denotes the Lebesgue space of functions on $\mathbb{U}^d$ equipped with the standard $L_p$-norm. For $s \in \mathbb{N}$, the space $W^s_p(\mathbb{U}^d)$ is defined analogously to $W^s_p(\mathbb{R}^d,\gamma)$ by replacing $L_p(\mathbb{R}^d,\gamma)$ in \eqref{eq:sobolev} with $L_p(\mathbb{U}^d)$.

For $s>0$ and $s \notin \mathbb{N}$, we write $s = \bar{s} + \tilde{s}$, where $\bar{s} := \lfloor s \rfloor$ is the integer part of $s$ and $\tilde{s}$ is the fractional part. 

\begin{definition}\label{def}
	Let $1\leq p<\infty$,  $s >0$ and $s\not \in \NN$. We define the fractional  Sobolev space $W_{p}^s(\UUd)$ as
	\[
	W_{p}^s(\UUd)
	:=
	\left\{
	f \in W^{\bar{s}}_{p}(\UUd)
	\; : \;
	[ f]_{W_{p}^s(\UUd)} < \infty
	\right\},
	\]
	where
	\begin{equation*} 
		[ f]_{W_{p}^s(\UUd)}
		:=
		\sum_{|\balpha|_1 = \bar{s}}
		\Bigg(
		\int_{\UUd}	\int_{\UUd}
		\frac{\big| D^\balpha f(\bx) -D^\balpha f(\by) \big|^p}{|\bx-\by|^{d+\tilde{s}p}}  
		\rd\bx \, \rd\by
		\Bigg)^{\frac{1}{p}}.	 
	\end{equation*} 
	The norm of $f\in W_{p}^s(\UUd)$ is defined by
	$$
	\|f\|_{W_{p}^s(\UUd)}:=\|f\|_{W^{\bar{s}}_{p}(\UUd)}+	[ f]_{W_{p}^s(\UUd)}.
	$$
\end{definition}
These types of spaces are sometimes referred to as Sobolev–Slobodeckij spaces, and they coincide with the Besov spaces $B^s_{p,p}(\mathbb{U}^d)$; see \cite[Theorem 2.5.12]{Tri83B}. Analogously to the above definition, we define the fractional Gaussian Sobolev spaces as follows.

\begin{definition} \label{def:Sobolev1}
	Let $1\leq p<\infty$,  $s >0$ and $s\not \in \NN$. We define the fractional Gaussian Sobolev space $W_{p,G}^s(\RRd,\gamma)$ as
	\[
	W_{p,G}^s(\RRd,\gamma)
	:=
	\left\{
	f \in W^{\bar{s}}_{p}(\RRd,\gamma)
	\; : \;
	[ f]_{W_{p,G}^s(\RRd,\gamma)} < \infty
	\right\},
	\]
	where
	\begin{equation*}\label{eq:def}
		[ f]_{W_{p,G}^s(\RRd,\gamma)}
		:=
		\sum_{|\balpha|_1 = \bar{s}}
		\Bigg(
		\int_{\RRd}	\int_{\RRd}
		\frac{\big| D^\balpha f(\bx) -D^\balpha f(\by) \big|^p}{|\bx-\by|^{d+\tilde{s}p}}  
		\rd \gamma(\bx) \, \rd\gamma(\by)
		\Bigg)^{\frac{1}{p}}.	 
	\end{equation*} 
	The norm of $f\in W_{p,G}^s(\RRd,\gamma)$ is defined by
	$$
	\|f\|_{W_{p,G}^s(\RRd,\gamma)}:=\|f\|_{W^{\bar{s}}_{p}(\RRd,\gamma)}+	[ f]_{W_{p,G}^s(\RRd,\gamma)}.
	$$
\end{definition}

%%%%%%%%%%%%%%%%%%%%%%%%

Next, we introduce fractional Gaussian Sobolev spaces defined via kernel 
related to the fractional Ornstein-Uhlenbeck operator. 
In order to do so, we introduce the Ornstein--Uhlenbeck semigroup and its generator $\Delta_\gamma$. Let $t > 0$. For $v \in L_1(\RRd,\gamma)$  the Ornstein--Uhlenbeck semigroup is defined as
	\[
	e^{t\Delta_\gamma} v(\bx)
	:=
	\int_{\RRd} M_t(\bx,\by)\, v(\by)\, \rd\gamma(\by),\ \ \bx \in \RRd,
	\]
	where $M_t(\bx,\by)$ is the Mehler kernel
	\[
	M_t(\bx,\by)
	:=
	\frac{1}{(1 - e^{-2t})^{d/2}}
	\exp\left(
	- \frac{e^{-2t}|\bx|^2 - 2e^{-t} \bx \cdot \by + e^{-2t}|\by|^2}{2(1 - e^{-2t})}
	\right).
	\]
Ornstein--Uhlenbeck operator in $\RRd$ is a second-order partial differential operator
given by
	\[
	\Delta_\gamma v(\bx) = \Delta v(\bx) - \bx \cdot \nabla v(\bx).
	\]
We refer the reader to \cite{LMP20} and the references therein for the main properties of $e^{t\Delta_\gamma}$ and $\Delta_\gamma$.
	
For $\sigma \in (0,1)$, we define the fractional Ornstein--Uhlenbeck operator by means of spectral decomposition via the Bochner subordination formula, i.e.,
\begin{equation}\label{eq:OU-operator}
\begin{aligned}
	(-\Delta_\gamma)^\sigma v(\bx)
	&:= \frac{1}{\Gamma(-\sigma)} \int_0^\infty \frac{e^{t\Delta_\gamma}v(\bx) - v(\bx)}{t^{\sigma+1}}\, \rd t \\
	&= \frac{1}{\Gamma(-\sigma)} \int_0^\infty \frac{\rd t}{t^{\sigma+1}} \int_{\RRd} M_t(\bx,\by)\,(v(\by) - v(\bx))\, \rd\gamma(\by) \\
	&= \frac{1}{\Gamma(-\sigma)} \int_{\RRd} (v(\bx) - v(\by))\, K_{2\sigma}(\bx,\by)\, \rd\gamma(\by),
\end{aligned}	 
\end{equation}
where
$\Gamma$ denotes the gamma function and
 \begin{equation*}
	K_\sigma(\bx,\by) := \int_0^\infty \frac{M_t(\bx,\by)}{t^{\frac{\sigma}{2}+1}}\, \rd t \,.
\end{equation*}
Note that the right-hand side in \eqref{eq:OU-operator} has to be understood in the Cauchy principal value sense.  

\begin{definition}\label{def:Sobolev2}
	Let $1\leq p<\infty$,  $s >0$ and $s\not \in \NN$. We define the fractional Gaussian Sobolev space $W_{p}^s(\mathbb{R}^d,\gamma)$ as
	\[
	W_{p}^s(\mathbb{R}^d,\gamma)
	:=
	\left\{
	f \in W^{\bar{s}}_{p}(\mathbb{R}^d,\gamma)
	\; ; \;
	[ f]_{W_{p}^s(\RRd,\gamma)} < \infty
	\right\},
	\]
	where
	\begin{equation*} 
		[ f]_{W_{p}^s(\mathbb{R}^d,\gamma)}
		:=
		\sum_{|\balpha|_1 = \bar{s}}
		\Bigg(
		\int_{\RRd}	\int_{\RRd}
		\big| D^\balpha f(\bx) -D^\balpha f(\by) \big|^p K_{p\tilde{s}}(\bx,\by)
		\rd\gamma(\bx) \, \rd\gamma(\by)
		\Bigg)^{\frac{1}{p}}.	 
	\end{equation*} 
	The norm of $f\in W_{p}^s(\mathbb{R}^d,\gamma)$ is defined by
	$$
	\|f\|_{W_{p}^s(\mathbb{R}^d,\gamma)}:=\|f\|_{W^{\bar{s}}_{p}(\mathbb{R}^d,\gamma)}+	[ f]_{W_{p}^s(\mathbb{R}^d,\gamma)}.
	$$
\end{definition}

In the following we collect some properties of the kernel $K_\sigma(\bx,\by)$ which is useful to formulate our results. 
\begin{lemma}\label{lem:help}For $\bx,\by\in \RRd$ and $\sigma>0$ we have
	$$ \frac{2^{\sigma+\frac{d}{2}}\Gamma( \frac{\sigma+d}{2})}{|\bx-\by|^{d+\sigma}} \leq 
	K_\sigma(\bx,\by) \,.
	$$
\end{lemma}
This lemma was proved in \cite[Lemma 2.8]{CCMP22}. We have some further estimates of the kernel $K_\sigma(\bx,\by)$.
\begin{lemma}  \label{lem:estimate}
	Let $\sigma>0$.
	\begin{enumerate}
		\item[(1)] If $\bx,\by\in [-\ell,\ell]^d$, $\ell>0$, then it holds
		$$
		K_\sigma(\bx,\by) \le \frac{C_\ell}{|\bx-\by|^{d+\sigma}}.
		$$
		\item[(2)] If $|\bx-\by|\geq M$, then it holds
$$
K_\sigma(\bx,\by) \le C_M \exp\bigg(\frac{|\bx|^2+|\by|^2}{4}\bigg).
$$
		\item[(3)] Let $\bell\in \RRd$ and $t_0>0$. If $\bx,\by\in [0,1]^d+\bell$, then
		$$
		K_\sigma(\bx,\by)\geq C_{d,t_0} 	\exp\bigg(\frac{|\bx|^2+|\by|^2}{4} \Big(1-\frac{ 1-e^{-t_0}}{ 1+e^{-t_0}}\Big)\bigg) \frac{1}{|\bx-\by|^{d+\sigma}} ,
		$$
where $C_{d,t_0}$ is a positive constant depending only on $d$ and $t_0$.
	\end{enumerate}
	
\end{lemma}
\begin{proof}
	By the definition of Mehler kernel we have
	\begin{equation}\label{eq:Mehler}
		\begin{aligned}
			M_t(\bx,\by)
			&=	\exp\bigg(\frac{|\bx|^2+|\by|^2}{4}\bigg(1-\frac{(1-e^{-t})}{(1+e^{-t})}\bigg)\bigg)\frac{1}{(1 - e^{-2t})^{d/2}}
			\exp\left(
			- \frac{e^{-t}|\bx-\by|^2}{2(1 - e^{-2t})}
			\right).
		\end{aligned}
	\end{equation}
	Observe that
	\[
	\frac{1}{(1 - e^{-2t})^{d/2}}
	\exp\left(
	- \frac{e^{-t}|\bx-\by|^2}{2(1 - e^{-2t})}
	\right)
	\sim \frac{1}{(2 t)^{d/2}} \exp\left(
	- \frac{|\bx-\by|^2}{4t}
	\right), \ \ t\to 0.
	\]
Therefore
	\begin{equation}\label{eq:K-sigma}
		\begin{aligned}
			{K}_\sigma(\bx,\by)
		&	=\int_0^1 \frac{M_t(\bx,\by)}{t^{\frac{\sigma}{2}+1}}\rd t +\int_1^\infty \frac{M_t(\bx,\by)}{t^{\frac{\sigma}{2}+1}}\rd t \\
		&\leq  C_1 \exp\bigg(\frac{|\bx|^2+|\by|^2}{4}\bigg) \bigg( \int_{1}^{\infty}  
			\frac{\rd t}{t^{\frac{\sigma}{2}+1}} +\int_0^{1}\frac{\exp\big(
				- \frac{|\bx-\by|^2}{4t}
				\big)}{(2 t)^{d/2}t^{\frac{\sigma}{2}+1}}\rd t
			\bigg)
			\\
			&\leq C_2 \exp\bigg(\frac{|\bx|^2+|\by|^2}{4}\bigg) \bigg(  1+  \int_0^{\infty}\frac{\exp\big(
				- \frac{|\bx-\by|^2}{4t}
				\big)}{t^{\frac{\sigma}{2}+\frac{d}{2}+1}}\rd t
			\bigg).
		\end{aligned}
	\end{equation}
Performing  change of variable in the second term on the right-hand side we get
$$
\int_0^{\infty}\frac{\exp\big(
	- \frac{|\bx-\by|^2}{4t}
	\big)}{t^{\frac{\sigma}{2}+\frac{d}{2}+1}}\rd t= \frac{2^{d+\sigma}\Gamma(\frac{d+\sigma}{2})}{|\bx-\by|^{d+\sigma}}.
$$
Inserting this into \eqref{eq:K-sigma} we get the first two statements. For the thirst one, from \eqref{eq:Mehler} if $t\in (0,t_0)$ and $|\bx-\by|\leq 1$, then we have
	\begin{equation*}
		\begin{aligned}
			M_t(\bx,\by)
			&\geq 		\exp\bigg(\frac{|\bx|^2+|\by|^2}{4} \Big(1-\frac{ 1-e^{-t_0}}{ 1+e^{-t_0}}\Big)\bigg)\frac{1}{(1 - e^{-2t})^{d/2}}
			\exp\left(
			- \frac{e^{-t}|\bx-\by|^2}{2(1 - e^{-2t})}
			\right)
			\\
			&\geq C_{t_0} 	\exp\bigg(\frac{|\bx|^2+|\by|^2}{4} \Big(1-\frac{ 1-e^{-t_0}}{ 1+e^{-t_0}}\Big)\bigg)\frac{1}{(2t)^{d/2}}
			\exp\left(
			- \frac{|\bx-\by|^2}{4t}
			\right).
		\end{aligned}
	\end{equation*}
Note that, the constant $C_{t_0}$ does not depend on $\bx,\by$.	Therefore, we obtain
	\begin{equation}\label{eq:the-third}
		\begin{aligned}
			K_\sigma(\bx,\by)&\geq  \int_0^{t_0} \frac{M_t(\bx,\by)}{t^{\frac{\sigma}{2}+1}} \rd t \\
			&\geq  C_{t_0} 	\exp\bigg(\frac{|\bx|^2+|\by|^2}{4} \Big(1-\frac{ 1-e^{-t_0}}{ 1+e^{-t_0}}\Big)\bigg) \int_0^{t_0}\frac{	\exp\left(
				- \frac{|\bx-\by|^2}{4t}\right)}{t^{\sigma/2+d/2+1}}
		\rd t\,.
		\end{aligned}
	\end{equation}
Performing a change of variables in the second term on the right-hand side and using $|\bx-\by|\leq \sqrt{d}$ we get
\begin{equation*}
\begin{aligned}
\int_0^{t_0}\frac{\exp\big(
	- \frac{|\bx-\by|^2}{4t}
	\big)}{t^{\frac{\sigma}{2}+\frac{d}{2}+1}}\rd t &= \frac{2^{\sigma+d}}{|\bx-\by|^{d+\sigma}}\int_{\frac{|\bx-\by|^2}{4t_0}}^{\infty}e^{-\xi} \xi^{\frac{\sigma}{2}+\frac{d}{2}-1}\rd\xi
	\\
	&\geq \frac{2^{\sigma+d}}{|\bx-\by|^{d+\sigma}}\int_{\frac{\sqrt{d}}{4t_0}}^{\infty}e^{-\xi} \xi^{\frac{\sigma}{2}+\frac{d}{2}-1}\rd\xi.	 
\end{aligned}
\end{equation*}
From this and \eqref{eq:the-third} we obtain the third estimate. 	The proof is completed.
	\hfill
\end{proof}

From the above lemmas we deduce the following results.
\begin{lemma}\label{lem:embedding}
	Let $1\leq p<\infty$.
	\begin{itemize}
		\item[(1)] If $s>0$ and $s\not \in \NN$, then  
		$W_{p}^s(\mathbb{R}^d,\gamma)$ is continuously embedded into  $W^s_{p,G}(\RRd,\gamma)$. 
		\item [(2)] Let $s>\frac{d}{p}$ or $s=d$ and $p=1$. Then $f\in W^s_{p}(\RRd,\gamma)$  is continuous on $\RRd$.
	\end{itemize}
\end{lemma}
\begin{proof} {\it (1)} If $f\in W^s_{p,G}(\RRd,\gamma)$, then by Lemma \ref{lem:help} we have
	$$
	[ f]_{W_{p,G}^s(\RRd,\gamma)} \leq C 	[ f]_{W_{p}^s(\RRd,\gamma)}
	$$
	which implies 
	$$
	\| f\|_{W_{p,G}^s(\RRd,\gamma)} \leq C 	\| f\|_{W_{p}^s(\RRd,\gamma)}
	$$
	and hence
	$$W_{p}^s(\mathbb{R}^d,\gamma)\hookrightarrow W^s_{p,G}(\RRd,\gamma).$$
This proves the first statement.
	\\
	{\it (2)} Let $f\in W^s_{p}(\RRd,\gamma)$. For any $\ell>0$, consider the restriction of $f$ to the cube $\UUd_\ell:=[-\ell,\ell]^d$, denoted by $f_\ell$. Then by Definitions \ref{def}, \ref{def:Sobolev2} and Lemma \ref{lem:help} we get $f_\ell\in \Wap(\UUd_\ell)$. It is well known that $\Wap(\UUd_\ell)=B^s_{p,p}(\UUd_\ell)$ is continuously embedded into the space of bounded continuous functions on $\UUd_\ell$, see, e.g., \cite{SiTr95}. Hence $f_\ell$ is continuous on $\UUd_\ell$. Since $\ell>0$ is arbitrary it follows that 
	$f$ is   continuous on $\RRd$.
	\hfill
\end{proof}
\begin{lemma}
	Let $s>0$ and $1\leq p,q < \infty$. 
\begin{itemize}
	\item[(1)]  If $1\leq q\leq p<\infty$ then we have the continuous embedding $\Wpgamma\to L_q(\RRd,\gamma)$.
\item[(2)] If $1\leq p<q< \infty$ then we  do not have have the continuous embedding, i.e., $\Wpgamma\not\hookrightarrow L_q(\RRd,\gamma)$
\end{itemize} 
\end{lemma}
\begin{proof}
In the first part of the lemma, the case $p=q$  is trivial. The case $q<p$ is due to $W^s_p(\RRd,\gamma)\hookrightarrow L_p(\RRd,\gamma)$ and $L_p(\RRd,\gamma)\hookrightarrow L_q(\RRd,\gamma)$ by the H\"older's inequality. For the second part, consider the function $$f(\bx)=e^{\frac{|\bx|^2}{2p}}\big(1+|\bx|^2\big)^{-m},\ \ \bx\in \RRd.$$ It is clear that this function belongs to $\Wpgamma$ if $m$ large enough. However, it does not belong to $L_q(\RRd,\gamma)$ when $q >p$ for any $m$. \hfill
\end{proof}

	For $k\in \NN_0$, the normalized probabilistic Hermite polynomial
$H_k$ of degree $k$ on $\RR$ is defined by
\begin{equation*}\label{eq:hermite} 
	H_k(x) 
	:= 
	\frac{(-1)^k}{\sqrt{k!}} 
	\exp\left(\frac{x^2}{2}\right) \frac{\rd^k}{\rd x^k} \exp\left(-\frac{x^2}{2}\right)  
\end{equation*}
and the $d$-variate Hermite
polynomial $H_\bk$ for   $\bk\in \NNd_0$  
\begin{equation*}\label{H_bk}
	H_\bk(\bx) :=\prod_{j=1}^d H_{k_j}(x_j),
	\;\; \bx\in \RRd.
\end{equation*}
It is well known that the Hermite polynomials $(H_\bk)_{\bk \in \NNd_0}$ form an orthonormal basis of the Hilbert space $L_2(\RRd,\gamma)$ (see, e.g.,  \cite[Section 5.5]{Szego1939}). In particular,  every function $f \in L_2(\RRd,\gamma)$ admits a representation in terms of its Hermite series
\begin{equation}\label{H-series}
	f = \sum_{\bk \in \NNd_0} \widehat{f}(\bk) H_\bk \ \ {\rm with} \ \ \widehat{f}(\bk) := \int_{\RRd} f(\bx)\, H_\bk(\bx)\rd \gamma(\bx) 
\end{equation}
with convergence in the norm of $L_2(\RRd,\gamma)$ and there holds the Parseval's identity
\begin{equation*}\label{P-id}
	\norm{f}{L_2(\RRd,\gamma)}^2= \sum_{\bk \in \NNd_0} |\widehat{f}(\bk)|^2.
\end{equation*}
\begin{definition}
	Let $s > 0$. The space $H^s(\mathbb{R}^d,\gamma)$ consists of all functions $f \in L_2(\mathbb{R}^d,\gamma)$ that admit a representation in terms of the Hermite series \eqref{H-series} and for which the norm
	\begin{equation*}
		\|f\|_{H^s(\mathbb{R}^d,\gamma)} := \Bigg(\sum_{\bk \in \mathbb{N}_0^d} (1+|\bk|_1)^{s}  |\widehat{f}(\bk)|^2 \Bigg)^{1/2}<\infty.
	\end{equation*}
\end{definition}
%%%%%%%%%%%%%%%%%%%%%%%%%%%%
We have the following assertion.
\begin{lemma}\label{lem:embedding2}
	Let $s>0$. Then $$	W_{2}^s(\mathbb{R}^d,\gamma)=	H^s(\mathbb{R}^d,\gamma)$$ in the sense of equivalent norms.
\end{lemma}
\begin{proof}
First we need the representation of derivatives of $f\in W^{s}_2(\RRd,\gamma)$, $s\in \NN$, in terms of Hermite polynomials  which was proved in \cite{DILP18}. For $\balpha\in \NN_0^d$ with $|\balpha|_1\leq s$ we have
\begin{equation}\label{eq:dick}
D^{\balpha}f(\bx)=\sum_{\bk\ge\balpha}\widehat{f}\left(\bk\right)\sqrt{\frac{\bk!}{\left(\bk-\balpha\right)!}}H_{\bk-\balpha}\left(\bx\right).
\end{equation}
Hence, we get for $\balpha\in \NN_0^d$ with $|\balpha|_1\leq s$
\begin{equation*}
\begin{aligned}
\int_{\RRd}|D^\balpha f(\bx)|^2\rd \gamma(\bx) = \sum_{\bk\geq \balpha} \frac{\bk!}{\left(\bk-\balpha\right)!}|\widehat{f}(\bk)|^2 \asymp \sum_{\bk\ge\balpha}  \big(k_1^{\alpha_1}\cdot\ldots\cdot k_d^{\alpha_d}\big)|\widehat{f}(\bk)|^2 ,
\end{aligned}
\end{equation*}
with the convention $0^0=1$.
This derives the case $s \in \NN$
$$
\|f\|_{W^{s}_{p}(\mathbb{R}^d,\gamma)}\asymp \Bigg(\sum_{\bk \in \mathbb{N}_0^d} (1+|\bk|_1)^{s}  |\widehat{f}(\bk)|^2 \Bigg)^{1/2} .
$$
Now we consider the case $s\not \in \NN$. We use the following formula for $v\in W^\sigma_2(\RRd,\gamma)$
\begin{equation}\label{eq:help}
	\begin{aligned}
		\int_{\mathbb{R}^d} \rd\gamma(\bx)
		\int_{\mathbb{R}^d}
		|v(\bx) - v(\by)|^2 \, K_{2\sigma}(\bx,\by)\, \rd\gamma(\by)&=
		2\Gamma(-\sigma)\int_{\mathbb{R}^d} v(\bx)\, (-\Delta_\gamma)^\sigma v(\bx) \, \rd\gamma(\bx)
	\end{aligned}	 
\end{equation}
which can be found in \cite[page 6]{CCMP22}. For $f\in W^s_{p}(\RRd,\gamma)$ we put
 $v=D^\balpha f$ with $\balpha\in \NN_0^d$ and $|\balpha|_1=\bar{s}$. By using \eqref{eq:help} we get
\begin{equation}\label{eq:quasi-norm}
	\begin{aligned}
				[ f]_{W_{2}^s(\mathbb{R}^d,\gamma)}^2&=	2\Gamma(-\tilde{s})\int_{\mathbb{R}^d} v(\bx)\, (-\Delta_\gamma)^{\tilde{s}} v(\bx) \, \rd\gamma(\bx)
				\\
				&=	2\Gamma(-\tilde{s})\sum_{\bk\in \NN_0^d}\widehat{v}(\bk)\int_{\mathbb{R}^d} v(\bx)\, (-\Delta_\gamma)^{\tilde{s}} H_\bk(\bx) \, \rd\gamma(\bx)
							\\
				&=	2\Gamma(-\tilde{s})\sum_{\bk\in \NN_0^d}\widehat{v}(\bk)\int_{\mathbb{R}^d} v(\bx)\, \big(k_1^{\tilde{s}}+\ldots+k_d^{\tilde{s}}\big) H_\bk(\bx) \, \rd\gamma(\bx)
				\\
							&=	2\Gamma(-\tilde{s})\sum_{\bk\in \NN_0^d} \big(k_1^{\tilde{s}}+\ldots+k_d^{\tilde{s}}\big) |\widehat{v}(\bk) |^2.
	\end{aligned}
\end{equation}
Employing \eqref{eq:dick} we get 
\begin{equation*}
	\begin{aligned}
\widehat{v}(\bk) &=\int_{\RRd} D^\balpha f(\bx)H_\bk (\bx) \rd\gamma(\bx)
\\
& = \sum_{\bell\ge\balpha}\int_{\RRd}\widehat{f}\left(\bell\right)\sqrt{\frac{\bell!}{\left(\bell-\balpha\right)!}}H_{\bell-\balpha}\left(\bx\right) H_\bk(\bx)\rd\gamma(\bx)
\\
&
 =\sqrt{\frac{(\balpha+\bk)!}{\bk!}}\widehat{f}(\balpha+\bk)	.
	\end{aligned}
\end{equation*}
Inserting this into \eqref{eq:quasi-norm} we finally get
\begin{equation*}
	[ f]_{W_{2}^s(\mathbb{R}^d,\gamma)}^2= 2\Gamma(-\tilde{s})\sum_{\bk\in \NN_0^d} \big(k_1^{\tilde{s}}+\ldots+k_d^{\tilde{s}}\big)   \frac{(\balpha+\bk)!}{\bk!} |\widehat{f}(\balpha+\bk) |^2.
\end{equation*}
From this the second statement follows. 
The proof is completed. \hfill
\end{proof}

\section{Approximation for Gaussian Sobolev spaces}\label{sec:numer}
In this section, we establish the asymptotic behavior of the Kolmogorov, linear, and sampling widths for functions in Gaussian Sobolev spaces $W^s_p(\mathbb{R}^d,\gamma)$, where the error is measured in the norm of $L_q(\mathbb{R}^d,\gamma)$. Building upon known linear algorithms for these quantities in Sobolev spaces defined on the domain $\mathbb{I}^d = [-1/2,1/2]^d$, we construct corresponding linear algorithms for functions in Gaussian Sobolev spaces $W^s_p(\mathbb{R}^d,\gamma)$ that achieve the same rate of convergence for  the case $1\le q < p<\infty$.

We first define Kolmogorov, linear and sampling widths. 
Let $Y$ be a Banach space and let $F \subset Y$ be a centrally symmetric compact set. For $n \in \mathbb{N}$, the Kolmogorov $n$-width of $F$ in $Y$ is defined by
\begin{equation*}
	d_n(F,Y)= \inf_{L_{n}}\sup_{f\in F}\inf_{g\in L_n}\|f-g\|_Y, 
\end{equation*}
where the infimum is taken over all linear subspaces $L_n \subset Y$ of dimension at most $n$.

The linear $n$-width of $F$ is defined by
$$
\lambda_n(F,Y):=\inf_{A_n} \sup_{f\in F} \|f-A_n(f)\|_Y,
$$
where the infimum is taken over all linear operators $A_n: X \to X$ with rank at most $ n$. 

Let $\Omega  $ be a domain in $\RRd$, and let $Y$ be a Banach space of functions defined on $\Omega$. Let $F \subset Y$ be a compact set and let $n \in \mathbb{N}$. Given sampling points $\{\bx_i\}_{i=1}^n \subset \Omega$, we aim to approximately recover a function $f \in F$ from the sampled values $\{f(\bx_i)\}_{i=1}^n$ by means of a linear sampling algorithm defined by
\begin{equation*} \label{R_n(f)}
	R_n(f): = \sum_{i=1}^n  f(\bx_i) \varphi_i,
\end{equation*}
where $\{\varphi_i\}_{i=1}^n$ is a family of functions in $Y$. For the sake of generality, we permit repetitions both in the set of sampling points $\{\bx_i\}_{i=1}^n$ and in the collection of functions $\{\varphi_i\}_{i=1}^n$.
For each $n \in \mathbb{N}$, the sampling $n$-width of the set $F$ in $Y$ is defined as 
$$
\varrho_n(F,Y):=\inf_{\bx_1,\ldots,\bx_n\in \Omega,\atop
	\varphi_1,\ldots,\varphi_n\in Y} \ \sup_{f\in F}
\|f- R_n(f)\|_Y.
$$
Obviously, we have the inequalities
\begin{equation*}\label{eq-relations}
	d_n(F,Y) \leq \lambda_n(F,Y)\leq \varrho_n(F,Y).
\end{equation*}
It is well known that if $Y$ is a Hilbert space, then $\lambda_n(F,Y) = d_n(F,Y)$.

Let $T$ be a linear bounded operator between two Banach spaces $X$ and $Y$. The $n$th Kolmogorov number of the linear operator $T$ is defined as 
\begin{equation*}
	d_n(T)= \inf_{L_{n-1}}\sup_{\|x\|_X\leq 1}\inf_{y\in L_{n-1}}\|Tx-y\|_Y. \label{def1}
\end{equation*}
Here the outer infimum is taken over all linear subspaces  $L_{n-1}$ of dimension ($n-1$)  in $Y$. The $n$th approximation number of the linear operator $T$ is defined as
	$$
	a_n(T):=\inf\big\{\|T-A\|: \ A\in \mathcal L(X,Y),\ \ \text{rank} (A)<n\big\}\, , \qquad n \in \NN\, . 
	$$
The Kolmogorov and approximation numbers belong to the class called $s$-numbers. If $B_X$ is the closed unit ball in $X$, then 
$$
d_{n+1}(T)=d_n(T(B_X),Y),\ \ a_{n+1}(T)=\lambda_n(T(B_X),Y).
$$ 
For a detailed exposition of the theory of widths and $s$-numbers, we refer the reader to the monographs \cite{Pin85B, Pie80B, Pie87B}.

%%%%%%%%%%%%%%%
 Denote by $\tilde{L}_q(\UUd)$ and $\tilde{W}^\alpha_p(\UUd)$ the subspaces of  $L_q(\UUd) $ and $W^s_{p}(\UUd)$, respectively,  of all functions $f$ which can be extended to the whole $\RRd$ as $\UU$-periodic  functions in each variable. The following proposition is crucial for the formulation of our main results.
%%%%%%%%%%%%%%%

\begin{prop} \label{prop:-general}
Let $1 \le q < p < \infty$ and  $a > 0$, $b \ge 0$. Let $s > 0$ in the case of Kolmogorov and linear widths, and $s > d/p$ in the case of sampling widths.
Assume that for each $m \in \mathbb{N}$, there exists a linear operator $A_m$ on $\tilde{L}_q(\mathbb{I}^d)$ of rank at most $m$ such that
\begin{equation}\label{A_m-Error-a,b-theta}
	\| g - A_m(g) \|_{\tilde{L}_q(\IId)}\leq C m^{-a} (\log m)^b \|g\|_{\tilde{W}^s_p(\IId)}, 
	\ \  g\in \tilde{W}^s_p(\IId).
\end{equation}
Then, for any $n \in \mathbb{N}$, one can construct, based on this operator, a linear operator $A_n^\gamma$ on $L_q(\mathbb{R}^d,\gamma)$ of rank at most $n$ such that
	\begin{equation*}\label{A_m-Error-a,b}
		\| f - A_{n}^\gamma(f) \|_{\Lqgamma}\leq C n^{-a} (\log n)^b \|f\|_{\Wpgamma}, 
		\ \  f\in \Wpgamma .
	\end{equation*}
Moreover, if $1 \le q < \frac{p}{2} < \infty$, then
 \begin{equation*} 
	\| f - A_{n}^\gamma(f) \|_{\Lqgamma}\leq C n^{-a} (\log n)^b \|f\|_{W^s_{p,G}(\RRd,\gamma)}, 
	\ \  f\in W^s_{p,G}(\RRd,\gamma). 
\end{equation*}
\end{prop}
\begin{proof} The proofs for the spaces $\Wpgamma$ and $W^s_{p,G}(\mathbb{R}^d,\gamma)$ are similar. We provide a detailed proof for the space $\Wpgamma$ and only comment on the proof for the space $W^s_{p,G}(\mathbb{R}^d,\gamma)$.
	
	For a fixed   $\theta>1 $ we denote the $d$-cube $\IId_\theta$ by $\IId_\theta := \big[-\frac{\theta}{2}, \frac{\theta}{2}\big]^d$, 
	$\IId_{\theta,\bk}:=\bk+\IId_\theta$  for $\bk \in \ZZd$, and
	$ f_{\bk} $ the restriction of $f$ on $\IId_{\theta,\bk}$  for a function $f$ on $\RRd$.  Let  $\kappa>\max\{s,d/p-s\}$, $\kappa\in \NN$  and $(\varphi_\bk)_{\bk \in \ZZd}$ be a partition of unity on $\RRd$ satisfying  
	\begin{itemize}
		\item[\rm{(i)}] $\varphi_\bk \in C^\infty_0(\RRd)$ and 
		$0 \le \varphi_\bk (\bx)\le 1$, \ \ $\bx \in \RRd$, \ \ $\bk \in \ZZd$;
		\item[\rm{(ii)}] the support of $\varphi_\bk$ is contained in the interior of  $\IId_{\theta,\bk}$, $\bk \in \ZZd$;
		\item[\rm{(iii)}]  $\sum_{\bk \in \ZZd}\varphi_\bk (\bx)= 1$, \ \ $\bx \in \RRd$;
		\item[\rm{(iv)}]  $\norm{\varphi_\bk }{C^\kappa(\RRd)} \le C_{\kappa,d,\theta}$,  
		$\bk \in \ZZd$.
	\end{itemize}
This implies that for $f\in W^s_{p}(\RRd,\gamma)$ we have
	\begin{align*}  
		f(\bx)
		=  \sum_{\bk \in \ZZd}f_{\bk}(\bx)\varphi_\bk(\bx).
	\end{align*}
Next we will prove that $f_\bk(\cdot+\bk)\varphi_\bk(\cdot+\bk)$ belongs to $\tilde{W}_p^{s}(\IId_\theta)$. First, we have
\begin{equation*}\label{eq:b4}
	\begin{split}
		\|f_{\bk}(\cdot+\bk)\|_{W_p^{\bar{s}}(\IId_\theta)}&=\Bigg(\sum_{|\balpha|_1 \leq \bar{s}} \|D^\balpha f_\bk(\cdot+\bk)\|_{L_p(\IId_\theta)}^p\Bigg)^{1/p}
		\\
		&=\Bigg(\sum_{|\balpha|_1 \leq \bar{s}} (2\pi)^{d/2} \int_{\IId_{\theta,\bk}}e^{\frac{|\bx|^2}{2}}|D^\balpha f_\bk(\bx)| ^p\rho(\bx)\rd\bx \Bigg)^{1/p}.
	\end{split}
\end{equation*}
Choose $q_*$ and $p_*$ such that $1\leq q<q_*<p_*<p $. When $\bx \in \IId_{\theta,\bk}$ we have 
$e^{\frac{|\bx|^2}{2}}\leq Ce^{\frac{p|\bk |^2}{2p_*}}$. Therefore,
\begin{equation}\label{eq:bar{s}}
	\|f_\bk(\cdot+\bk)\|_{W_p^{\bar{s}}(\IId_\theta)}
	\leq  Ce^{\frac{|\bk|^2}{2p_*}}\|f\|_{W^s_{p}(\RRd,\gamma)}.
\end{equation}
We choose $t_0>0$ such that
$$
\frac{1}{2p_*}\bigg(1+\frac{ 1-e^{-t_0}}{ 1+e^{-t_0}}\bigg)-\frac{1}{2q_*}<0.
$$
This is possible since $q_*<p_*$. Let $\delta>0$ such that
\begin{equation}   \label{eq:delta}
	\bigg\{	e^{\frac{|\bk |^2}{2p_*} \big(1+\frac{ 1-e^{-t_0}}{ 1+e^{-t_0}}\big)-\frac{|\bk|^2}{2q_*}}, e^{- \frac{q|\bk|^2}{2q_*}\big(1-\frac{q}{p}\big)}\bigg\}
	\leq 
	C e^{-\delta |\bk|^2}.
\end{equation}
Using Lemma \ref{lem:estimate} (3), with $\balpha\in \NN_0^d$ such that $|\balpha|_1\leq \bar{s}$ we get
\begin{equation*}
	\begin{aligned}
		&\int_{\IId_\theta}\int_{\IId_\theta}
		\frac{
			\big| D^\balpha f_\bk(\bx+\bk)  -D^\balpha f_\bk(\by+\bk)\big|^p
		}{
		|\bx-\by|^{d+\tilde{s}p}
		}
		\rd\bx	\rd\by
		\\
		&\leq  C_{d,t_0}\int_{\IId_{\theta,\bk}}\int_{\IId_{\theta,\bk}}
		\big| D^\balpha f_\bk(\bx)-D^\balpha f_\bk(\by) \big|^p e^{-\frac{|\bx|^2+|\by|^2}{4} \big(1-\frac{ 1-e^{-t_0}}{ 1+e^{-t_0}}\big)} K_{p\tilde{s}}(\bx,\by)
		\rd\bx	\rd\by
		\\
		&\leq  C_{d,t_0}\int_{\IId_{\theta,\bk}}\int_{\IId_{\theta,\bk}} 
		\big| D^\balpha f_\bk(\bx) -D^\balpha f_\bk(\by)\big|^p e^{\frac{|\bx|^2+|\by|^2}{4} \big(1+\frac{ 1-e^{-t_0}}{ 1+e^{-t_0}}\big)}K_{p\tilde{s}}(\bx,\by)
	 \rd\gamma(\bx)	\rd\gamma(\by)
		\\
		&\leq   C_{d,t_0}  e^{\frac{p|\bk |^2}{2p_*} \big(1+\frac{ 1-e^{-t_0}}{ 1+e^{-t_0}}\big)} \|f\|_{\Wpgamma}^p.
	\end{aligned}
\end{equation*}

From this and \eqref{eq:bar{s}} we find that
$$
\|f_\bk(\cdot+\bk)\|_{W^s_{p}(\IId_\theta)}\leq C_{d,t_0} e^{\frac{|\bk|^2}{2p_*} \big(1+\frac{ 1-e^{-t_0}}{ 1+e^{-t_0}}\big)} \|f\|_{\Wpgamma}.
$$	
By pointwise multiplication of $W^s_{p}(\IId_\theta)$  \cite[Theorem 2.8.2]{Tri83B},  we have  
\begin{equation*}  
	\begin{aligned} 
		\|f_{\bk}(\cdot+\bk)\varphi_{\bk}(\cdot+\bk)\|_{W^s_{p}(\IId_\theta)} 
		& 
		\leq C_{t_0} \|f_{\bk}(\cdot+\bk)\|_{W^s_{p}(\IId_\theta)}  \norm{\varphi_\bk(\cdot+\bk) }{C^\kappa(\RRd)}
		\\
		&	\leq C_{d,t_0} e^{\frac{|\bk |^2}{2p_*} \big(1+\frac{ 1-e^{-t_0}}{ 1+e^{-t_0}}\big)}\|f\|_{\Wpgamma}.
	\end{aligned}
\end{equation*}

  We define	for $n\in \NN$,
\begin{equation} \label{xi-int}	
	\xi_n =  \sqrt{\delta^{-1} 2 a(\log n)}\,,
\end{equation}
and for $\bk \in \ZZd$ with $|\bk|< \xi_n$
\begin{equation*} \label{n_bk}
	n_{\bk}= 
	\lfloor \varrho n  e^{-\frac{\delta}{2 a}|\bk|^2}  \rfloor,
\end{equation*}
where $\delta$ is in \eqref{eq:delta} and \begin{equation}\label{eq-c2-2.23}
	\varrho:=\Bigg(\sum_{j=0}^{\infty}\left[\left(2j+1\right)^{d}-\left(2j-1\right)^{d}\right]e^{-\frac{\delta}{2a}j^{2}}\Bigg)^{-1}.
\end{equation} We have 
$$
\sum_{\left|\bk\right|<\xi_{n}} n_{\bk} \leq \sum_{\left|\bk\right|<\xi_{n}} \varrho n e^{-\frac{\delta}{2 a}\left|\bk\right|^{2}}\leq \varrho n \sum_{\left|\bk\right|_{\infty}<\xi_{n}} e^{-\frac{\delta}{2 a}\left|\bk\right|_{\infty}^{2}}\leq \varrho n \sum_{j=0}^{\left\lfloor\xi_{n}\right\rfloor}\sum_{\left|\bk\right|_{\infty}=j}e^{-\frac{\delta}{2 a}j^2}.
$$
Since
$$|\{\bk \in \mathbb{Z}^{d}: \left|\bk\right|_{\infty}=j\}|=\left(2j+1\right)^{d}-\left(2j-1\right)^{d},$$ from
\eqref{eq-c2-2.23} we get
\begin{equation}\label{eq:<n}
	\sum_{\left|\bk\right|<\xi_{n}} n_{\bk} \leq \varrho n \sum_{j=0}^{\left\lfloor\xi_{n}\right\rfloor}\left[\left(2j+1\right)^{d}-\left(2j-1\right)^{d}\right]e^{-\frac{\delta}{2 a}j^{2}} \leq n.
\end{equation}
By \eqref{A_m-Error-a,b-theta} we have that
\begin{equation*}
	\| g - A_{\theta,m}(g) \|_{\tilde{L}_q(\IId_\theta)}\leq C m^{-a} (\log m)^b \|g\|_{\tilde{W}^\alpha_p(\IId_\theta)}, 
	\ \  g\in \tilde{W}^\alpha_p(\IId),
\end{equation*}
where $A_{\theta,m}$ is the induced operator of $A_{m}$ to the cube $\IId_\theta$.
From this and $(f_\bk \varphi_\bk)(\cdot+\bk)\in W^s_p(\IId_\theta)$  
 we derive the estimates
\begin{align*}
&\norm{(f_{\bk}\varphi_\bk)  - A_{\theta, n_\bk}\big[f_\bk\varphi_\bk(\cdot +\bk)\big](\cdot-\bk)}
{L_q(\IId_{\theta,\bk}, \gamma)}  
	\\
	& \ll  e^{- \frac{|\bk  |^2} {2q_*}}
	\norm{f_{\bk}\varphi_\bk - A_{\theta, n_\bk}\big[f_\bk\varphi_\bk(\cdot +\bk)\big](\cdot-\bk)}
	{\tilde{L}_q(\IId_{\theta,\bk})} 
	\\
& = e^{- \frac{|\bk  |^2} {2q_*}}
\norm{f_{\bk}\varphi_\bk(\cdot+\bk) - A_{\theta, n_\bk}\big[f_\bk\varphi_\bk(\cdot +\bk)\big]}
{\tilde{L}_q(\IId_{\theta,\bk})} 
	\\
	& \ll  e^{- \frac{|\bk  |^2} {2q_*}}  n_{\bk}^{-a} (\log n_{\bk})^b \|(f_\bk\varphi_\bk)(\cdot+\bk)\|_{\tilde{W}^\alpha_p(\IId_\theta)}
	\\
	&
	\ll  e^{\frac{|\bk|^2}{2p_*} \big(1+\frac{ 1-e^{-t_0}}{ 1+e^{-t_0}}\big)-
		\frac{|\bk  |^2} {2q_*}}
	\Big( n e^{-\frac{\delta}{2a}|\bk|^2} \Big)^{-a} (\log n)^b \|f\|_{\Wpgamma}.
\end{align*}
By the choice of $\delta$ in \eqref{eq:delta} we obtain
\begin{align*}
\norm{(f_{\bk}\varphi_\bk)  - A_{\theta, n_\bk}\big[f_\bk\varphi_\bk(\cdot +\bk)\big](\cdot-\bk)}
{L_q(\IId_{\theta,\bk}, \gamma)}  
	&
	\ll  e^{- \frac{\delta}{2} |\bk|^2}  n^{-a}  (\log n)^b  \|f\|_{\Wpgamma},
\end{align*}
which implies
\begin{equation}\label{eq:first-term}
\begin{aligned}
	\sum_{|\bk|< \xi_n}& \norm{(f_{\bk}\varphi_\bk)  - A_{\theta, n_\bk}\big[f_\bk\varphi_\bk(\cdot +\bk)\big](\cdot-\bk)}
	{L_q(\IId_{\theta,\bk}, \gamma)}  
	\\
	&\ll \sum_{|\bk|< \xi_n}   e^{- \frac{ \delta}{2} |\bk|^2} 
	n^{-a}  (\log n)^b  \|f\|_{\Wpgamma}
	\\&
	\ll  n^{-a}  (\log n)^b  \|f\|_{\Wpgamma}.
\end{aligned}	 
\end{equation}

We define the operator $A^\gamma_{n}$ as follows
$$
A^\gamma_{n}(f)=\sum_{|\bk|<\xi_n} A_{\theta, n_\bk}\big[f_\bk\varphi_\bk(\cdot +\bk)\big](\cdot-\bk).
$$
By \eqref{eq:<n} we have
$
\rank(A^\gamma_{n}) \leq n 
$
and
\begin{equation}\label{eq:main}
	 \begin{aligned}  
	 	\|f- A_{n}^\gamma(f)\|_{\Lqgamma}
	 	&\leq 
	 	\sum_{|\bk|< \xi_n} \norm{(f_{\bk}\varphi_\bk)  - A_{\theta, n_\bk}\big[f_\bk\varphi_\bk(\cdot +\bk)\big](\cdot-\bk)}
	 	{L_q(\IId_{\theta,\bk}, \gamma)}  
	 	\\
	 	&
	 	+ \sum_{|\bk|\geq \xi_n} 	\norm{f_{\bk}\varphi_\bk}{L_q(\IId_{\theta,\bk}, \gamma)}. 
	 \end{aligned}
\end{equation}
Observe that
\begin{equation*}
\begin{aligned}
		\sum_{|\bk|\geq \xi_n} 	\norm{f_{\bk}\varphi_\bk}{L_q(\IId_{\theta,\bk}, \gamma)}
		& \leq  
			\sum_{|\bk|\geq \xi_n} 	\norm{f_{\bk}\varphi_\bk}{L_p(\IId_{\theta,\bk}, \gamma)} \bigg(\int_{\IId_{\theta,\bk}} \rd \gamma(\bx)\bigg)^{1-\frac{q}{p}}
			\\
		& \ll
		\sum_{|\bk|\geq \xi_n} 
		e^{- \frac{q|\bk|^2}{2q_*}\big(1-\frac{q}{p}\big)}
		\|f\|_{\Wpgamma}.
	\end{aligned}
\end{equation*} 
By the choice of $\delta$ in \eqref{eq:delta} we obtain
\begin{equation*}
	\begin{aligned}
		\sum_{|\bk|\geq \xi_n} 	\norm{f_{\bk}\varphi_\bk}{L_q(\IId_{\theta,\bk}, \gamma)}
		&	\ll  \sum_{|\bk|\geq \xi_n}  e^{- \delta |\bk|^2}  \|f\|_{\Wpgamma}
		\\
		&
		\leq \sum_{\ell=\lceil\xi_n^2\rceil}^\infty   \sum_{|\bk|^2=\ell}e^{-\ell \delta}\|f\|_{W^s_{p}(\RRd,\gamma)}
		\\
		&  \leq \sum_{\ell=\lceil\xi_n^2\rceil}^\infty   \sum_{|\bk|_\infty\leq \sqrt{\ell}}e^{-\ell \delta}\|f\|_{W^s_{p}(\RRd,\gamma)}
		\\
		&  \ll \sum_{\ell=\lceil\xi_n^2\rceil}^\infty   e^{-\ell \delta}\ell^{d/2} \|f\|_{W^s_{p}(\RRd,\gamma)}.
	\end{aligned}
\end{equation*} 
Fix $\varepsilon \in (0,1/2)$ we get
\begin{equation*}\label{eq-epsilon01}
	\begin{aligned}
	\sum_{|\bk|\geq \xi_n} 	\norm{f_{\bk}\varphi_\bk}{L_q(\IId_{\theta,\bk}, \gamma)}
		&
		\leq  e^{-\xi_n^2\delta(1-\varepsilon)}\sum_{\ell=\lceil\xi_n^2\rceil}^\infty   e^{-\ell\varepsilon \delta}\ell^{d/2}\|f\|_{W^s_{p}(\RRd,\gamma)}
		\\
		&
		\ll e^{-\xi_n^2\delta(1-\varepsilon)}\|f\|_{W^s_{p}(\RRd,\gamma)}.
	\end{aligned}
\end{equation*} 
Using \eqref{xi-int} we arrive at
\begin{equation}\label{eq:second-term}
	\begin{aligned}
	\sum_{|\bk|\geq \xi_n} 	\norm{f_{\bk}\varphi_\bk}{L_q(\IId_{\theta,\bk}, \gamma)}	&  \ll  e^{-2a(1-\varepsilon)\log n}\|f\|_{W^s_{p}(\RRd,\gamma)}		
		\ll n^{-a}  (\log n)^b\|f\|_{W^s_{p}(\RRd,\gamma)}.
	\end{aligned}
\end{equation}
By combining \eqref{eq:first-term}, \eqref{eq:main}, and \eqref{eq:second-term}, the proof for the space $W^s_p(\mathbb{R}^d,\gamma)$ is complete.
We comment on the proof for $W^s_{p,G}(\RRd,\gamma)$. 
Choose $q_*$ and $p_*$ such that $1\leq q<q_*<\frac{p_*}{2}<\frac{p}{2} $.   
Observe that 
\begin{equation*}
	\begin{aligned}
		&\int_{\IId_\theta}\int_{\IId_\theta}
		\frac{
			\big| D^\balpha f_\bk(\bx+\bk)  -D^\balpha f_\bk(\by+\bk)\big|^p
		}{
			|\bx-\by|^{d+\tilde{s}p}
		}
		\rd\bx	\rd\by
		\\
		&\ll \int_{\IId_{\theta,\bk}}\int_{\IId_{\theta,\bk}} 
		\big| D^\balpha f_\bk(\bx) -D^\balpha f_\bk(\by)\big|^p e^{\frac{|\bx|^2+|\by|^2}{2}  } \rd\gamma(\bx)	\rd\gamma(\by)
		\\
		&\ll    e^{\frac{p|\bk |^2}{p_*} } \|f\|_{W^s_{p,G}(\RRd,\gamma)}
	\end{aligned}
\end{equation*}
which together with \eqref{eq:bar{s}} implies 
\begin{equation*}  
	\begin{aligned} 
		\|f_{\bk}(\cdot+\bk)\varphi_{\bk}(\cdot+\bk)\|_{W^s_{p}(\IId_\theta)} 
		& \ll e^{\frac{|\bk |^2}{p_*}} \|f\|_{W^s_{p,G}(\RRd,\gamma)}.
	\end{aligned}
\end{equation*}	
Therefore
\begin{align*}
	\norm{(f_{\bk}\varphi_\bk)  - A_{\theta, n_\bk}\big[f_\bk\varphi_\bk(\cdot +\bk)\big](\cdot-\bk)}
	{L_q(\IId_{\theta,\bk}, \gamma)}  
	\leq   e^{\frac{|\bk |^2}{p_*}-
		\frac{|\bk  |^2} {2q_*}}
   n^{-a}  (\log n)^b  \|f\|_{W^s_{p,G}(\RRd,\gamma)}
\end{align*}
and
\begin{align*}
 \sum_{|\bk|< \xi_n}& \norm{(f_{\bk}\varphi_\bk)  - A_{\theta, n_\bk}\big[f_\bk\varphi_\bk(\cdot +\bk)\big](\cdot-\bk)}
	{L_q(\IId_{\theta,\bk}, \gamma)}  
	\\
	&\ll \sum_{|\bk|< \xi_n}    e^{\frac{|\bk |^2}{p_*}-
		\frac{|\bk  |^2} {2q_*}}
	n^{-a}  (\log n)^b  \|f\|_{W^s_{p,G}(\RRd,\gamma)}.
\end{align*}
The condition $q_* < \frac{p_*}{2}$ ensures that the sum is bounded by $C n^{-a} (\log n)^b \|f\|_{W^s_{p,G}(\mathbb{R}^d,\gamma)}$. The remaining steps follow similarly to those in the proof for $\Wpgamma$. The proof is complete.
\hfill
\end{proof}

Let us now recall the asymptotic behavior of Kolmogorov, linear and sampling widths for the Sobolev–Slobodeckij spaces on $\IId$ which coincide with the Besov spaces $B^s_{p,p}(\IId)$.
\begin{theorem}\label{thm:recal}
	Let $1\leq q<p<\infty$.  
\begin{itemize}
	\item[(1)] If $s>0$, then $$d_n\big(\tilde{\boldsymbol{W}^s_{p}}(\IId),\tilde{L}_q(\IId) \big)
	\asymp \lambda_n\big(\tilde{\boldsymbol{W}^s_{p}}(\IId),\tilde{L}_q(\IId) \big) 
\asymp	n^{-s/d}. $$
\item[(2)] If $s>d/p$, then
$$  \varrho_n\big(\tilde{\boldsymbol{W}^s_{p}}(\IId),\tilde{L}_q(\IId) \big)\asymp	n^{-s/d}.$$
\end{itemize}
\end{theorem}
The results for Kolmogorov and linear widths has a long history. We refer the reader to \cite{EdmundsTriebel1996,Cae1998,Vy2008} and the references therein. The result for sampling widths was proved in \cite{NT06}. Combining the above theorem with Proposition~\ref{prop:-general}, we can state one of our main results as follows.
\begin{theorem}\label{thm:main}
	Let $1\leq q<p<\infty$.  
	\begin{itemize}
		\item[(1)] If $s>0$, then $$	  d_n\big({\boldsymbol{W}}^s_{p}(\mathbb{R}^d,\gamma),L_q(\RRd,\gamma )\big)
		  \asymp \lambda_n\big({\boldsymbol{W}}^s_{p}(\mathbb{R}^d,\gamma),L_q(\RRd,\gamma) \big)
	 	\asymp n^{-s/d}.  $$
		\item[(2)] If $s>d/p$, then
		$$  	   \varrho_n\big({\boldsymbol{W}}^s_{p}(\mathbb{R}^d,\gamma),L_q(\RRd,\gamma) \big)	\asymp n^{-s/d}. $$
	\end{itemize}
\end{theorem}
\begin{proof} We present the proof for the case of Kolmogorov and linear widths; the case of sampling widths follows similarly. In this proof, for $k>0$ we denote $\UUd_k=[-k,k]^d$. 
Let $f$ be a $1$-periodic function on $\RRd$ and $f\in  \tilde{W}^\alpha_p(\IId)$. Denote by $f_{\ext}$ the extension of $f$ to $W^s_{p}(\RRd)$. For the existence of such a extension we refer the reader to \cite[Theorem 1.105]{Tri06B}. Without loss of generality we can assume that $\supp(f)\in \UUd_1$. Hence, we have
$$
\|f_{\ext}\|_{W^s_p(\RRd)} \leq C\|f\|_{\tilde{W}^s_{p}(\IId)}.
$$
We have
\begin{align*}
	\|f_{\ext}\|_{W^{\bar{s}}_p(\RRd,\gamma)}&= \Bigg( \sum_{|\balpha|_1 \leq \bar{s}} \int_{\RRd} |D^\balpha f_{\ext}(\bx)|^p  \rd \gamma(\bx)\Bigg)^{1/p}
	\\
	& \leq \Bigg( \sum_{|\balpha|_1 \leq \bar{s}} \int_{\RRd} |D^\balpha f_{\ext}(\bx)|^p  \rd  \bx\Bigg)^{1/p}
	\\
	&
	\leq 
	\|f_{\ext}\|_{W^s_p(\RRd)} \leq C\|f\|_{\tilde{W}^s_p(\IId)}.
\end{align*}
Next we estimate the term $[ f_{\ext}]_{W_{p}^s(\mathbb{R}^d,\gamma)}$. We have
	\begin{equation*} 
\begin{aligned}
[ f_{\ext}]_{W_{p}^s(\mathbb{R}^d,\gamma)}
&:=
\sum_{|\balpha|_1 = \bar{s}}
\Bigg(  	\int_{\RRd}	\int_{\RRd}
\big| D^\balpha f_{\ext}(\bx) -D^\balpha f_{\ext}(\by) \big|^pK_{p\tilde{s}}(\bx,\by)
\rd\gamma(\bx)   \rd\gamma(\by)
\Bigg)^{\frac{1}{p}}.	 
\end{aligned}
\end{equation*} 
 Note that if $\bx,\by\not \in \UUd_1$ then the integral is zero. If $\bx\in \UUd_1$ and $\by\in \RRd\backslash \UUd_2$ then by Lemma \ref{lem:estimate} (2) we get
\begin{equation*}
\begin{aligned}
&\Bigg(  	\int_{\RRd\backslash \UUd_2}	\int_{\UUd_1}
\big| D^\balpha f_{\ext}(\bx) -D^\balpha f_{\ext}(\by) \big|^pK_{p\tilde{s}}(\bx,\by)
\rd\gamma(\bx)   \rd\gamma(\by)
\Bigg)^{\frac{1}{p}}
\\
&\ll\Bigg(  	\int_{\RRd\backslash \UUd_2}	\int_{\UUd_1}
\big| D^\balpha f_{\ext}(\bx) \big|^p\exp\Big(\frac{|\bx|^2+|\by|^2}{4}\Big)
\rd\gamma(\bx)   \rd\gamma(\by)
\Bigg)^{\frac{1}{p}}
\\
&\ll\|f_{\ext}\|_{W^{\bar{s}}_p(\RRd)}.
\end{aligned}
\end{equation*}  Therefore  $[ f_{\ext}]_{W_{p}^s(\mathbb{R}^d,\gamma)}$ can be estimated as 
	\begin{equation*} 
	\begin{aligned}
	&	[ f_{\ext}]_{W_{p}^s(\mathbb{R}^d,\gamma)}
		\ll\|f_{\ext}\|_{W^{\bar{s}}_p(\RRd)}
		\\
		& +
		\sum_{|\balpha|_1 = \bar{s}}
		\Bigg(  	\int_{\UUd_2}	\int_{\UUd_2}
		\big| D^\balpha f_{\ext}(\bx) -D^\balpha f_{\ext}(\by) \big|^pK_{p\tilde{s}}(\bx,\by)
		\rd\gamma(\bx)   \rd\gamma(\by)
		\Bigg)^{\frac{1}{p}}.
	\end{aligned}
\end{equation*}
Now applying Lemma \ref{lem:estimate} (1) we get
	\begin{equation*} 
	\begin{aligned}
		[ f_{\ext}]_{W_{p}^s(\mathbb{R}^d,\gamma)}
	&	\ll\|f_{\ext}\|_{W^{\bar{s}}_p(\RRd)}
	+
		\sum_{|\balpha|_1 = \bar{s}}
		\Bigg(  	\int_{\UUd_2}	\int_{\UUd_2}
		\frac{	\big| D^\balpha f_{\ext}(\bx) -D^\balpha f_{\ext}(\by) \big|^p}{|\bx-\by|^{d+\tilde{s}p}} 
		\rd\bx  \rd\by
		\Bigg)^{\frac{1}{p}}
		\\
		&\leq  \|f_{\ext}\|_{W^{\bar{s}}_p(\RRd)}	 +[f_{\ext}]_{W^s_p(\RRd)}
		\\
		& = \|f_{\ext}\|_{W^s_{p}(\RRd)} 
		\\
		&\leq C \|f\|_{\tilde{W}^s_{p}(\RRd)}.
	\end{aligned}
\end{equation*}
We also have
$$
\|f\|_{\tilde{L}_q(\IId)}={\bigg((2\pi)^{\frac{d}{2}}\int_{\IId}|f(\bx)|^qe^{\frac{|\bx|^2}{2}} \rd \gamma( \bx)\bigg)^{1/q} }\leq (2\pi)^{\frac{d}{2q}}e^{\frac{d}{8q}}\|f\|_{\Lqgamma}.
$$
Hence, we obtain
\begin{align*}  
\lambda_n\big({\boldsymbol{W}}^s_{p}(\mathbb{R}^d,\gamma),L_q(\RRd,\gamma) \big) \geq 	  d_n\big({\boldsymbol{W}}^s_{p}(\mathbb{R}^d,\gamma),L_q(\RRd,\gamma )\big)
  \gg	d_n(\tilde{\boldsymbol{W}^s_p}(\IId),\tilde{L}_q(\IId)).
\end{align*}	
Now Theorem \ref{thm:recal} implies the lower bounds. The upper bound follows from Theorem \ref{thm:recal} and Proposition \ref{prop:-general}.	\hfill
\end{proof}

By using Proposition~\ref{prop:-general}, Theorem \ref{thm:recal} and arguing as in the proof of Theorem~\ref{thm:main}, we obtain the following result.

\begin{theorem}\label{thm:main21}
	Let $1\leq q<\frac{p}{2}<\infty$. \begin{itemize}
		\item[(1)] If $s>0$, then $$d_n\big({\boldsymbol{W}}^s_{p,G}(\RRd,\gamma),L_q(\RRd,\gamma) \big)
		\asymp \lambda_n\big({\boldsymbol{W}}^s_{p,G}(\RRd,\gamma),L_q(\RRd,\gamma) \big) 
		\asymp	n^{-s/d}. $$
		\item[(2)] If $s>d/p$, then
		$$  \varrho_n\big({\boldsymbol{W}}^s_{p,G}(\RRd,\gamma),L_q(\RRd,\gamma) \big)\asymp	n^{-s/d}.$$
	\end{itemize}
\end{theorem}
We proceed with the result in the case $p=q=2$.
\begin{theorem} 	\label{theorem:widths:p=q=2}
	Let $s >0$ and $s_n\in \{\lambda_n,d_n\}$. Then 
	\begin{equation*}\label{widths:p=q=2}
		s_n\big({\boldsymbol{W}}^s_{2}(\mathbb{R}^d,\gamma),L_2(\RRd,\gamma )\big)
		\asymp s_n\big({\boldsymbol{H}}^s(\mathbb{R}^d,\gamma),L_2(\RRd,\gamma) \big)
		\asymp 
		n^{-\frac{s}{2d}}.
	\end{equation*}
	Moreover
\begin{equation*} \label{eq-asymptotic}
	\lim\limits_{n\to \infty} \frac{s_n\big({\boldsymbol{H}}^s(\mathbb{R}^d,\gamma),L_2(\RRd,\gamma) \big)}{n^{-\frac{s}{2d}}}= \bigg( \frac{1}{\sqrt[d]{d!}}\bigg)^{s/2}.
\end{equation*}
\end{theorem}
\begin{proof}
	Since   $L_2(\RRd,\gamma)$ is a Hilbert space
	$$	d_n\big({\boldsymbol{W}}^s_{2}(\mathbb{R}^d,\gamma),L_2(\RRd,\gamma )\big)
	= \lambda_n\big({\boldsymbol{W}}^s_{2}(\mathbb{R}^d,\gamma),L_2(\RRd,\gamma) \big)$$
as well as
	$$	d_n\big({\boldsymbol{H}}^s(\mathbb{R}^d,\gamma),L_2(\RRd,\gamma )\big)
= \lambda_n\big({\boldsymbol{H}}^s(\mathbb{R}^d,\gamma),L_2(\RRd,\gamma) \big).$$
By Lemma \ref{lem:embedding2} we get
$$s_n\big({\boldsymbol{W}}^s_{2}(\mathbb{R}^d,\gamma),L_2(\RRd,\gamma )\big)
\asymp s_n\big({\boldsymbol{H}}^s(\mathbb{R}^d,\gamma),L_2(\RRd,\gamma) \big).$$
	Consider the diagram
	\begin{equation*}
		\begin{CD}
			H^s(\RRd,\gamma)  @ > Id >> L_2(\RRd,\gamma) \\
			@VV A V @AA B A\\
			\ell_2(\NNd_0) @ > D  >> \ell_2(\NNd_0) \,, 
		\end{CD}
	\end{equation*}
	where the linear operators $A$, 	$D$, and $B $ 
	are defined as 
	\begin{equation*}	\label{ws-12}
		\begin{aligned}
			Af & : =  \big( (1+|\bk|_1)^{s/2}\widehat{f}(\bk)\big)_{\bk\in \NNd_0}\, , \qquad f\in H^s(\RRd,\gamma)
			\\
			D\xi & :=   \big(\xi_\bk/ (1+|\bk|_1)^{s/2}\big)_{\bk\in \NNd_0}\,  , \qquad \xi=(\xi_\bk)_{\bk\in \NNd_0}
			\\
			(B\xi)(\bx)& :=  \sum_{\bk\in \NNd_0} \xi_\bk\,  H_\bk(\bx)\, , \qquad \bx \in  \RRd\, .
		\end{aligned}
	\end{equation*}
	It is obvious that $\|A\|=\|B\|=1$. It follows
	\[
d_{n-1}\big({\boldsymbol{H}}^s(\mathbb{R}^d,\gamma),L_2(\RRd,\gamma )\big)=d_n(id) \le d_n \big(D\big) \, .
	\]
	
	It is easily  seen that the operators $A$ and $B$ are invertible and  that $\|A^{-1}\|=\|B^{-1}\|=1$. 
	Employing the same type of arguments with respect to the diagram
	\begin{equation*}
		\begin{CD}
			H^s(\RRd,\gamma)  @ > Id >> L_2(\RRd,\gamma) \\
			@AA {A^{-1}} A @VV B^{-1} V\\
			\ell_2(\NNd_0) @ > D >> \ell_2(\NNd_0) \, 
		\end{CD}
	\end{equation*} we also get
	\[
d_{n-1}\big({\boldsymbol{H}}^s(\mathbb{R}^d,\gamma),L_2(\RRd,\gamma )\big) =d_n(Id)\geq d_n \big(D\big) \, .
	\]
	Consequently, we obtain
	\[
d_{n-1}\big({\boldsymbol{H}}^s(\mathbb{R}^d,\gamma),L_2(\RRd,\gamma )\big)=	d_n\big(Id\big)=
d_n\big(D\big) .
	\]
We know that $	d_n\big(D\big)=a_n(D)=\sigma_n$ where $\sigma_n$ is the non-increasing rearrangement of $(1+|\bk|_1)^{-s/2}$. For $r\in \NN$, we know that the cardinality of the set 
$
\{\bk\in \NNd_0: 1+|\bk|_1\leq r \}=
$ is $c(r,d)=\binom{r-1+d}{d}$. If $c(r-1,d) < n\leq c(r,d) $. Then we have
$$
r^{-s/2} \leq d_{c(r,d)}(D) \leq d_n(D) \leq d_{c(r-1,d)} \leq (r-1)^{-s/2}.
$$
Hence 
$$
\frac{c(r-1,d)}{r^d}\leq \frac{d_n(D)^{2d/s}}{n^{-1}} \leq \frac{c(r,d)}{(r-1)^d}.
$$
Using $c(r,d)=\binom{r-1+d}{d}$  we obtain the desired results.	\hfill
\end{proof}
\section{Gaussian Sobolev spaces with mixed smoothness}\label{sec:sec4}

This section is devoted to establishing the asymptotic behavior of the Kolmogorov and linear widths for functions in fractional Gaussian Sobolev spaces of mixed smoothness. The analysis builds upon the methodology developed in Section~3, suitably adapted to the present setting. In particular, we extend the underlying techniques to accommodate mixed smoothness and derive the corresponding asymptotic estimates.

	For $s \in \NN$, we define the Gaussian Sobolev  space with mixed smoothness $W^s_{p,\mix}(\RRd,\gamma)$ as the normed space of all functions $f\in L_p(\RRd,\gamma)$ such that the  generalized  partial derivatives $D^\balpha f$ of order $\balpha$  belong to $L_p(\RRd,\gamma)$ for all $\balpha\in \NN_0^d$ satisfying $|\balpha|_\infty\leq s$. The norm of a  function $f$ in this space 
is defined by
\begin{align*} \label{W-Omega}
	\|f\|_{W^s_{p,\mix}(\RRd,\gamma)}: = \Bigg(\sum_{|\balpha|_\infty \leq s} \|D^\balpha f\|_{L_p(\RRd,\gamma)}^p\Bigg)^{1/p}.
\end{align*}

For function $f$ defined on a set $\UUd\subset \RRd$ and $\bx,\by\in \UUd$ we denote
$$
\Delta_{y_j,j} f(\bx):=f(x_1,\ldots,x_{j-1},x_j,x_{j+1},\ldots,x_d)-f(x_1,\ldots,x_{j-1},y_j,x_{j+1},\ldots,x_d) .
$$ 
For $e\subset \{1,\ldots,d\}$, $e\not=\emptyset$,  we denote   $\by_e=(y_j)_{j\in e}\in \RR^{|e|}$ and
$$
\Delta^e_{\by_e}f(\bx):=\prod_{j\in e} \Delta_{y_j,j} f(\bx), 
$$
where $\Delta_{y_j,j}$ is the univariate operator applied to the $j$-th coordinate of $f$
with the other variables kept fixed. 
\begin{definition}\label{def:Sobolev}
	Let $1\leq p<\infty$, $s>0$ and $s\not \in \NN$. Then the fractional Sobolev spaces with mixed smoothness $W_{p,\mix}^{s}(\UUd)$ is defined by
	$$
	W^s_{p,\mix}(\UUd):=\big\{f\in W_{p,\mix}^{\bar{s}}(\UUd): [f]_{W_{p,\mix}^{s}(\UUd)} <\infty \big\}
	,$$
	where
	$$
	[f]_{W^s_{p,\mix}(\UUd)}
	:=\sum_{|\balpha|_\infty = \bar{s}}\sum_{e\subset [d],e\not=\emptyset}
	\Bigg(
	\int_{\UU^{|e|}}\int_{\UUd}
	\frac{
		\big| \Delta^e_{\by_e}D^\balpha f(\bx)  \big|^p
	}{
		\prod_{j\in e}	|x_j - y_j|^{1 + \tilde{s}p}
	}
	\, \rd\bx \, \rd\by_e
	\Bigg)^{\frac{1}{p}}.
	$$
	The norm of $f\in W_{p,\mix}^{s}(\UUd)$ is given by
	$$
	\|f \|_{W_{p,\mix}^{s}(\UUd)}:
	=
	\|f \|_{ W_{p,\mix}^{\bar{s}}(\UUd)}
	+ [f]_{W_{p,\mix}^{s}(\UUd)}.
	$$
\end{definition}
Note that the space $W_{p,\mix}^{s}(\UUd)$ coincides with the so-called Besov space of dominating mixed smoothness, often denoted by $S^s_{p,p}B(\UUd)$, see \cite[Theorems 2.2.6.2 and 2.3.4.1]{ST87B}.

\begin{definition}
	Let $1\leq p<\infty$,  $s >0$ and $s\not \in \NN$. We define the fractional Gaussian Sobolev space  with mixed smoothness $W_{p,\mix}^s(\mathbb{R}^d,\gamma)$ as
	\[
	W_{p,\mix}^s(\mathbb{R}^d,\gamma)
	:=
	\left\{
	f \in W^{\bar{s}}_{p,\mix}(\mathbb{R}^d,\gamma)
	\; : \;
	[ f]_{W_{p,\mix}^s(\RRd,\gamma)} < \infty
	\right\},
	\]
	where
	\begin{equation*}\label{eq:quasi-norm2}
		[ f]_{W_{p,\mix}^s(\mathbb{R}^d,\gamma)}
		:=
		\sum_{|\balpha|_\infty = \bar{s}}\sum_{e\subset [d],e\not=\emptyset}
		\Bigg(
		\int_{\RRd}	\int_{\RR^{|e|}}
		\big| \Delta^e_{\by_e}D^\balpha f(\bx)  \big|^2 \bigg(\prod_{i\in e}K_{p\tilde{s}}(x_i,y_i)\bigg)
		\rd\gamma(\bx) \, \rd\gamma(\by_e)
		\Bigg)^{\frac{1}{p}}.	 
	\end{equation*} 
	The norm of $f\in W_{p,\mix}^s(\mathbb{R}^d,\gamma)$ is defined by
	$$
	\|f\|_{W_{p,\mix}^s(\mathbb{R}^d,\gamma)}:=\|f\|_{W^{\bar{s}}_{p,\mix}(\mathbb{R}^d,\gamma)}+	[ f]_{W_{p,\mix}^s(\mathbb{R}^d,\gamma)}.
	$$
\end{definition}
We denote by $\tilde{W}^s_{p,\mix}(\mathbb{U}^d)$ the subspace of $W^s_{p,\mix}(\mathbb{U}^d)$ consisting of functions that are $\UU$-periodic in each variable.
The proof of the following proposition proceeds analogously to that of Proposition~\ref{prop:-general}.
\begin{prop} \label{prop:main2}
Let $s>0$, $1 \le q < p < \infty$ and let $a > 0$, $b \ge 0$.  
Assume that for each $m \in \mathbb{N}$, there exists a linear operator $A_m$ on $\tilde{L}_q(\mathbb{I}^d)$ of rank at most $m$ such that
	\begin{equation*} 
		\| g - A_m(g) \|_{\tilde{L}_q(\IId)}\leq C m^{-a} (\log m)^b \|g\|_{\tilde{W}^s_{p,\mix}(\IId)}, 
		\ \  g\in \tilde{W}^s_{p,\mix}(\IId).
	\end{equation*}
Then, for any $n \in \mathbb{N}$, one can construct, based on this operator, a linear operator $A_n^\gamma$ on $L_q(\mathbb{R}^d,\gamma)$ of rank at most $n$ such that
	\begin{equation*} 
		\| f - A_{n}^\gamma(f) \|_{\Lqgamma}\leq C n^{-a} (\log n)^b \|f\|_{W^s_{p,\mix}(\RRd,\gamma)}, 
		\ \  f\in W^s_{p,\mix}(\RRd,\gamma).
	\end{equation*}
\end{prop}
We recall the following result.
\begin{theorem}\label{thm:recal2}
	Let $s>0$ and $s\not \in \NN$. Assume that $2\leq q<p<\infty$ or $1\leq q<2\leq p<\infty$. Then we have
	\begin{equation*}  d_n\big({\boldsymbol{W}}^s_{p,\mix}(\IId),L_q(\IId) \big)
		\asymp \lambda_n\big({\boldsymbol{W}}^s_{p,\mix}(\IId),L_q(\IId) \big) 
		\asymp	n^{-s}(\log n)^{(d-1)(s+\frac{1}{2}-\frac{1}{p})}. 
	\end{equation*}
\end{theorem}
The lower bound for the Kolmogorov widths was proved by Galeev \cite{Galeev01}. The upper bound for linear width was obtained by using the hyperbolic cross operator, see, e.g., \cite{Romanyuk1991}. To the knowledge of the author, the asymptotic behavior of Kolmogorov and linear widths of the set ${\boldsymbol{W}}^s_{p,\mix}(\IId)$ in $L_q(\RRd,\gamma)$ in the case $1\leq q<p<2$ is  open. We refer the reader to the book \cite[Chapter 4]{DTU18B} for further comments. 
Using the same argument as in the proof of Theorem \ref{thm:main} we get the following.
\begin{theorem}\label{thm:main3}
	Let $s>0$ and $s\not \in \NN$. Assume that $2\leq q<p<\infty$ or $1\leq q<2\leq p<\infty$. Then we have
	\begin{equation*}
		\begin{aligned}
			d_n\big({\boldsymbol{W}}^s_{p,\mix}(\mathbb{R}^d,\gamma),L_q(\RRd,\gamma )\big)
			& \asymp \lambda_n\big({\boldsymbol{W}}^s_{p,\mix}(\mathbb{R}^d,\gamma),L_q(\RRd,\gamma) \big)
				\asymp n^{-s}(\log n)^{(d-1)(s+\frac{1}{2}-\frac{1}{p})}. 
		\end{aligned} 
	\end{equation*}
\end{theorem}
\noindent
{\bf Acknowledgments:} A part of this work was done when  the author was working at the Vietnam Institute for Advanced Study in Mathematics (VIASM). He would like to thank  the VIASM for providing a fruitful research environment and working condition. 
%\bibliographystyle{abbrv}
%\bibliography{AllBib}
%\end{document}

\end{document}